\documentclass[12pt]{amsart}
\usepackage{amsmath,amssymb,amsfonts,amsthm,amstext,graphicx,xcolor,enumerate}
\usepackage{cleveref,url,mathtools}
\usepackage{tikz}
\usepackage[numbers,sort&compress]{natbib}
\usepackage{epstopdf}

\usepackage[margin=2.7cm]{geometry}

\usetikzlibrary{shapes}
\usetikzlibrary{arrows}
\usetikzlibrary{positioning}
\usetikzlibrary{matrix}
\usetikzlibrary{decorations.markings}
\usetikzlibrary{backgrounds}

\title[Perfect snake-in-the-box codes]
{Perfect snake-in-the-box codes \\ for rank modulation}

\newtheorem{thm}{Theorem}
\newtheorem*{thm*}{Theorem}
\newtheorem{prop}[thm]{Proposition}
\newtheorem{lemma}[thm]{Lemma}
\newtheorem{cor}[thm]{Corollary}

\crefname{thm}{Theorem}{Theorems}
\crefname{lemma}{Lemma}{Lemmas}
\crefname{prop}{Proposition}{Propositions}
\crefname{cor}{Corollary}{Corollaries}
\crefname{section}{Section}{Sections}
\crefname{figure}{Figure}{Figures}
\crefname{table}{Table}{Tables}

\newcommand{\A}{\mathcal A}
\newcommand{\G}{\mathcal G}
\renewcommand{\S}{\mathcal S}
\newcommand{\s}{{*}}
\DeclareMathOperator{\id}{id}
\DeclareMathOperator{\order}{order}

\newcommand{\T}{\Delta}

\colorlet{darkgreen}{green!60!black}

\tikzset{midarrow/.style={decoration={
  markings,
  mark=at position .5 with {\arrow{>}}},postaction={decorate}}}

\newcounter{mycount}
\newenvironment{ilist}{\begin{list}{\rm(\roman{mycount})}%
   {\usecounter{mycount}\labelwidth=4mm\itemsep 3pt\leftmargin 1cm\rightmargin 17mm}}{\end{list}}

\newcommand{\df}[1]{\textbf{\boldmath #1}}

\author{Alexander E.\ Holroyd}
\address{Alexander E.\ Holroyd, Microsoft Research,
1 Microsoft Way, Redmond, WA 98052, USA} \email{holroyd at microsoft.com}
\urladdr{\url{http://research.microsoft.com/~holroyd/}}

\date{24 February 2016 (revised 10 October 2016)}

\begin{document}
\begin{abstract}
For odd $n$, the alternating group on $n$ elements is
generated by the permutations that jump an element from
any odd position to position $1$. We prove Hamiltonicity
of the associated directed Cayley graph for all odd
$n\neq 5$.  (A result of Rankin implies that the graph is
not Hamiltonian for $n=5$.)  This solves a problem
arising in rank modulation schemes for flash memory.  Our
result disproves a conjecture of Horovitz and Etzion, and
proves another conjecture of Yehezkeally and Schwartz.
\end{abstract}
\maketitle

\section{Introduction}

The following questions are motivated by applications involving flash memory.
Let $S_n$ be the \df{symmetric group} of permutations
$\pi=[\pi(1),\ldots,\pi(n)]$ of $[n]:=\{1,\ldots,n\}$, with composition
defined by $(\pi\rho)(i)=\pi(\rho(i))$.  For $2\leq k \leq n$ let
$$\tau_k:=\bigl[k,\;1,2,\ldots,k-1,\;k+1,\ldots,n\bigr]\in S_n$$
be the permutation that jumps element $k$ to position $1$ while shifting
elements $1,2,\dots,k-1$ right by one place. Let $\S_n$ be the \df{directed
Cayley graph} of $S_n$ with generators $\tau_2,\ldots,\tau_n$, i.e.\ the
directed graph with vertex set $S_n$ and a directed edge, labelled $\tau_i$,
from $\pi$ to $\pi\tau_i$ for each $\pi\in S_n$ and each $i=2,\ldots,n$.

We are concerned with self-avoiding directed cycles (henceforth referred to
simply as \df{cycles} except where explicitly stated otherwise) in $\S_n$. (A
cycle is self-avoiding if it visits each vertex at most once).  In
applications to flash memory, a permutation represents the relative ranking
of charges stored in $n$ cells. Applying $\tau_i$ corresponds to the
operation of increasing the $i$th charge to make it the largest, and a cycle
is a schedule for visiting a set of distinct charge rankings via such
operations.  Schemes of this kind were originally proposed in \cite{jiang}.

One is interested in maximizing the length of such a cycle, since this
maximizes the information that can be stored.  It is known that $\S_n$ has a
directed \df{Hamiltonian} cycle, i.e.\ one that includes \emph{every}
permutation exactly once; see
e.g.~\cite{johnson,holroyd-ruskey-williams,jiang,knuth-4-2}. However, for the
application it is desirable that the cycle should not contain any two
permutations that are within a certain fixed distance $r$ of each other, with
respect to some metric $d$ on $S_n$. The motivation is to avoid errors
arising from one permutation being mistaken for another
\cite{jiang,mazumdar}. The problem of maximizing cycle length for given $r,d$
combines notions of Gray codes~\cite{savage} and error-detecting/correcting
codes~\cite{baylis}, and is sometimes known as a snake-in-the-box problem.
(This term has its origins in the study of analogous questions involving
binary strings as opposed to permutations; see e.g.~\cite{snake}).

The main result of this article is that, in the case that has received most
attention (described immediately below)
there is a cycle that is \df{perfect},
i.e.\ that has the maximum size even among arbitrary sets of
permutations satisfying the distance constraint.

More precisely, our focus is following case considered in
\cite{yehezkeally-schwartz,horovitz-etzion,zhang-ge}. Let $r=1$ and let $d$
be the \df{Kendall tau} metric \cite{kendall}, which is defined by setting
$d(\pi,\sigma)$ to be the inversion number of $\pi^{-1}\sigma$, i.e.\ the
minimum number of elementary transpositions needed to get from $\pi$ to
$\sigma$. (The $i$th elementary transposition swaps the permutation elements
in positions $i$ and $i+1$, where $1\leq i\leq n-1$). Thus, the cycle is not
allowed to contain any two permutations that are related by a single
elementary transposition. The primary object of interest is the maximum
possible length $M_n$ of such a directed cycle in $\S_n$.

It is easy to see that $M_n\leq n!/2$.  Indeed, any set of permutations
satisfying the above distance constraint includes at most one from the pair
$\{\pi,\pi\tau_2\}$ for every $\pi$, but these pairs partition $\S_n$.  To
get a long cycle, an obvious approach is to restrict to the \df{alternating
group} $A_n$ of all even permutations. Since an elementary transposition
changes the parity of a permutation, this guarantees that the distance
condition is satisfied. The generator $\tau_k$ lies in $A_n$ if and only if
$k$ is odd. Therefore, if $n$ is odd, this approach reduces to the problem of
finding a maximum directed cycle in the directed Cayley graph $\A_n$ of $A_n$
with generators $\tau_3,\tau_5,\ldots,\tau_n$.  Yehezkeally and Schwartz
\cite{yehezkeally-schwartz} conjectured that for odd $n$ the maximum cycle
length $M_n$ is attained by a cycle of this type; our result will imply this.
(For even $n$ this approach is less useful, since without using $\tau_n$ we
can access only permutations that fix $n$.)  As in
\cite{yehezkeally-schwartz,horovitz-etzion,zhang-ge}, we focus mainly on odd
$n$.

For small odd $n$, it is not too difficult to find cycles in $\A_n$ with
length reasonably close to the upper bound $n!/2$, by ad-hoc methods. Finding
systematic approaches that work for all $n$ is more challenging.  Moreover,
getting all the way to $n!/2$ apparently involves a fundamental obstacle, but
we will show how it can be overcome.

Specifically, it is obvious that $M_3=3!/2=3$. For general
odd $n\geq 5$, Yehezkeally and Schwartz
\cite{yehezkeally-schwartz} proved the inductive bound
$M_n\geq n(n-2)M_{n-2}$, leading to $M_n\geq
\Omega(n!/\sqrt n)$ asymptotically.  They also showed by
computer search that $M_5=5!/2-3=57$. Horowitz and Etzion
\cite{horovitz-etzion} improved the inductive bound to
$M_n\geq (n^2-n-1)M_{n-2}$, giving $M_n=\Omega(n!)$.  They
also proposed an approach for constructing a longer cycle
of length $n!/2 - n+2(=(1-o(1))n!/2)$, and showed by
computer search that it works for $n=7$ and $n=9$. They conjectured
that this bound is optimal for all odd $n$.  Zhang and Ge
\cite{zhang-ge} proved that the scheme of
\cite{horovitz-etzion} works for all odd $n$, establishing
$M_n\geq n!/2-n+2$, and proposed another scheme aimed at
improving the bound by $2$ to $n!/2 - n+4$. Zhang and Ge
proved that their scheme works for $n=7$, disproving the
conjecture of \cite{horovitz-etzion} in this case, but were
unable to prove it for general odd $n$.

The obvious central question here is whether there exists a perfect cycle,
i.e.\ one of length $n!/2$, for any odd $n>3$. As mentioned above, Horovitz
and Etzion \cite{horovitz-etzion} conjectured a negative answer for all such
$n$, while the authors of \cite{zhang-ge,yehezkeally-schwartz} also speculate
that the answer is negative. We prove a \emph{positive} answer for $n\neq 5$.

\begin{thm}\label{main}
For all odd $n\geq 7$, there exists a directed Hamiltonian
cycle of the directed Cayley graph $\A_n$ of the
alternating group $A_n$ with generators
$\tau_3,\tau_5,\ldots,\tau_n$.  Thus, $M_n=n!/2$.
\end{thm}

Besides being the first of optimal length, our cycle has a somewhat simpler
structure than those in \cite{horovitz-etzion,zhang-ge}.  It may in
principle be described via an explicit rule that specifies which generator
should immediately follow each permutation $\pi$, as a function of $\pi$.
(See \cite{holroyd-ruskey-williams,williams} for other Hamiltonian cycles of
Cayley graphs that can be described in this way). While the improvement from
$n!/2-n+2$ to $n!/2$ is in itself unlikely to be important for applications,
our methods are quite general, and it is hoped that they will prove useful
for related problems.

We briefly discuss even $n$.  Clearly, one approach is to simply leave the
last element of the permutation fixed, and use a cycle in $\A_{n-1}$, which
gives $M_n\geq M_{n-1}$ for even $n$. Horovitz and Etzion
\cite{horovitz-etzion} asked for a proof or disproof that this is optimal. In
fact, we expect that one can do much better.  We believe that $M_n\geq
(1-o(1)) n!/2$ asymptotically as $n\to\infty$ (an $n$-fold improvement over
$(n-1)!/2$), and perhaps even $M_n\geq n!/2-O(n^2)$. We will outline a
possible approach to showing bounds of this sort, although it appears that a
full proof for general even $n$ would be rather messy. When $n=6$ we use this
approach to show $M_6\geq 315=6!/2-45$, improving the bound $M_6\geq 57$ of
\cite{horovitz-etzion} by more than a factor of $5$.

Hamiltonian cycles of Cayley graphs have been extensively studied, although
general results are relatively few. See
e.g.~\cite{pak-radoicic,curran-gallian,witte-gallian,knuth-4-2} for surveys.
In particular, it is unknown whether every \emph{undirected} Cayley graph is
Hamiltonian. Our key construction (described in the next section) appears to
be novel in the context of this literature also.

Central to our proof are techniques having their origins in change ringing
(English-style church bell ringing).  Change ringing is also concerned with
self-avoiding cycles in Cayley graphs of permutations groups (with a
permutation representing an order in which bells are rung), and change
ringers discovered key aspects of group theory considerably before
mathematicians did -- see~e.g.\ \cite{white,thompson,griffiths,tintin}. As we
shall see, the fact that $\A_5$ has no Hamiltonian cycle (so that we have the
strict inequality $M_5<5!/2$) follows from a theorem of Rankin
\cite{rankin,swan} that was originally motivated by change ringing.

\section{Breaking the parity barrier}
\label{parity}

In this section we explain the key obstruction that frustrated the previous
attempts at a Hamiltonian cycle of $\A_n$ in
\cite{yehezkeally-schwartz,horovitz-etzion,zhang-ge}. We then explain how it
can be overcome.  We will then use these ideas to prove \cref{main} in
\cref{hypergraphs,cycle}.

By a \df{cycle cover} of a directed Cayley graph we mean a
set of self-avoiding directed cycles whose vertex sets
partition the vertex set of the graph. A cycle or a cycle
cover can be specified in several equivalent ways: we can
list the vertices or edges encountered by a cycle in order,
or we can specify a starting vertex of a cycle and list the
generators it uses in order, or we can specify which
generator immediately follows each vertex -- i.e.\ the
label of the unique outgoing edge that belongs to the cycle
or cycle cover. It will be useful to switch between these
alternative viewpoints.

A standard approach to constructing a Hamiltonian cycle is to start with a
cycle cover, and then successively make local modifications that unite
several cycles into one, until we have a single cycle.  (See
\cite{rapaport-strasser,pak-radoicic,holroyd-ruskey-williams,williams,griffiths,
yehezkeally-schwartz,horovitz-etzion,zhang-ge,curran-gallian,
compton-williamson,witte-gallian} for examples.) However, in $\A_n$ and many
other natural cases, there is a serious obstacle involving parity, as we
explain next.

The \df{order} $\order(g)$ of a group element $g$ is the smallest $t\geq 1$
such that $g^t=\id$, where $\id$ is the identity.  In our case, let
$\tau_k,\tau_\ell$ be two distinct generators of $\A_n$, and observe that
their ratio $\rho:=\tau_\ell\tau_k^{-1}$ is simply the permutation that jumps
element $\ell$ to position $k$ while shifting the intervening elements by
$1$. For example, when $n=9$ we have $\tau_9=[912345678]$ and
$\tau_7^{-1}=[234567189]$, so $\tau_9\tau_7^{-1}=[123456978]$ (element $9$
jumps first to position $1$ and then back to position $7$).  In general, the
ratio $\rho$ has order $q:=|k-\ell|+1$, which is odd.  In the example, $q=3$.

The fact that $\order(\rho)=q$ corresponds to the fact that in the Cayley
graph $\A_n$, starting from any vertex, there is a cycle of length $2q$
consisting of directed edges oriented in \emph{alternating} directions and
with alternating labels $\tau_\ell$ and $\tau_k$.  Consider one such
alternating cycle $Q$, and suppose that we have a cycle cover that includes
all $q$ of the $\tau_k$-edges of $Q$.  Consequently, it includes none of the
$\tau_\ell$-edges of $Q$ (since it must include only one outgoing edge from
each vertex). An example is the cycle cover that uses the outgoing
$\tau_k$-edge from every vertex of $\A_n$. Then we may modify the cycle cover as
follows: delete all the $\tau_k$-edges of $Q$, and
add all the $\tau_\ell$-edges of $Q$. This results in a new
cycle cover, because each vertex of the graph still has exactly one incoming
edge and one outgoing edge present.

Suppose moreover that all the $\tau_k$-edges of $Q$ lay in distinct
cycles in the original cycle cover.  Then the effect of the modification is
precisely to unite these $q$ cycles into one new cycle (having the same
vertices).  The new cycle alternately traverses the new $\tau_\ell$-edges and
the remaining parts of the $q$ original cycles.  All other cycles of the
cycle cover are unaffected.  See \cref{qset3} (left) for the case
$(k,\ell)=(n-2,n)$ (with $q=3$), and \cref{qset3} (right) for the
permutations at the vertices of the alternating cycle $Q$.

A modification of the above type reduces the total number of cycles in the
cycle cover by $q-1$, and therefore, since $q$ is odd, it does not change the
\emph{parity of the total number of cycles}. Less obviously, it turns out
that this parity is preserved by such a modification even if we relax the
assumption that the $q$ deleted edges lie in distinct cycles.  (See
\cite{rankin} or \cite{swan} for proofs.)
 This is a problem, because many cycle covers that one might
naturally start with have an \emph{even} number of cycles.  This holds in
particular for the cycle cover that uses a single generator $\tau_k$
everywhere (for $n\geq 5$), and also for the one that arises in an obvious
inductive approach to proving \cref{main} (comprising
$|A_n|/|A_{n-2}|=n(n-1)$ cycles each of length $|A_{n-2}|$). Thus we can
(apparently) never get to a Hamiltonian cycle (i.e.\ a cycle cover of one
cycle) by this method.
\begin{figure}
\centering
\begin{minipage}{.5\textwidth}
\centering
\begin{tikzpicture}[scale=.75]
\path[use as bounding box] (-3.1,-3.3) rectangle (3.1,2.6);
 \foreach \i in {0,...,5}
 { \draw (60*\i:1) node[circle,fill,inner sep=1.7pt] (\i){};};
 \draw[very thick,darkgreen] (120:1) node[circle, draw,inner sep=3pt]{};
 \draw (120:1) node[circle,fill=darkgreen,inner sep=1.7pt]{};
 \draw[very thick,midarrow,blue,dotted] (0) to node[above right=-4pt and -2pt] {$\tau_{n-2}$} (1);
 \draw[very thick,midarrow,blue,dotted] (2) to node[above left=-4pt and -2pt] {$\tau_{n-2}$} (3);
 \draw[very thick,midarrow,blue,dotted] (4) to node[below] {$\tau_{n-2}$} (5);
 \draw[ultra thick,midarrow] (0) to node[below right=-4pt and -2pt] {$\tau_n$} (5);
 \draw[ultra thick,midarrow] (2) to node[above] {$\tau_n$} (1);
 \draw[ultra thick,midarrow] (4) to node[below left=-4pt and -2pt] {$\tau_n$} (3);
 \draw[very thick,midarrow,blue] (1) to[bend left=135,looseness=10] (0);
 \draw[very thick,midarrow,blue] (3) to[bend left=135,looseness=10] (2);
 \draw[very thick,midarrow,blue] (5) to[bend left=135,looseness=10] (4);
\end{tikzpicture}
\end{minipage}%
\begin{minipage}{.4\textwidth}\centering
\centering
\begin{tikzpicture}
\boldmath
\newcommand{\bdot}{\boldsymbol\cdot}
  \matrix (T) [matrix of math nodes, nodes in empty cells,
  row 1/.style={nodes={text=darkgreen}},
  row 7/.style={nodes={text=darkgreen}},
  font=\fontsize{13}{13}\bfseries\selectfont,column sep={12pt,between origins},row
    sep={14pt,between origins}] {
   \bdot&\bdot&\bdot&\bdot&\bdot&\bdot & a & b & c \\
  c & \bdot&\bdot&\bdot&\bdot&\bdot&\bdot & a & b \\
   \bdot&\bdot&\bdot&\bdot&\bdot&\bdot & c & a & b \\
  b & \bdot&\bdot&\bdot&\bdot&\bdot&\bdot & c & a \\
   \bdot&\bdot&\bdot&\bdot&\bdot&\bdot & b & c & a \\
  a & \bdot&\bdot&\bdot&\bdot&\bdot&\bdot & b & c \\
   \bdot&\bdot&\bdot&\bdot&\bdot&\bdot & a & b & c \\
 };
  \draw (T-1-9.center) ..controls(T-1-9.south)and(T-2-1.north).. (T-2-1.center);
  \draw[blue,very thick,dotted] (T-2-1.center)..controls(T-2-1.south)and(T-3-7.north)..(T-3-7.center);
  \draw (T-3-9.center)..controls(T-3-9.south)and(T-4-1.north)..(T-4-1.center);
  \draw[blue,very thick,dotted] (T-4-1.center)..controls(T-4-1.south)and(T-5-7.north)..(T-5-7.center);
  \draw (T-5-9.center)..controls(T-5-9.south)and(T-6-1.north)..(T-6-1.center);
  \draw[blue,very thick,dotted] (T-6-1.center)..controls(T-6-1.south)and(T-7-7.north)..(T-7-7.center);
\end{tikzpicture}

\end{minipage}
 \caption{\emph{Left:} linking $3$ cycles by replacing generator $\tau_{n-2}$ with
 generator $\tau_n$ in $3$ places.  We start with the $3$ thin blue cycles,
 each of which comprises
 a dotted edge labeled with generator $\tau_{n-2}$, and a curved arc that
 represents the remaining part of the cycle.  We delete the dotted edges
 and replace them with the thick solid black edges (labelled $\tau_n$),
 to obtain one (solid) cycle, containing the same vertices as the original $3$ cycles.
 \emph{Right:} the permutations at the six vertices
 that are marked with solid discs in the left picture.
 The permutation at the (green) circled vertex is $[\ldots\ldots,a,b,c]$, where
 $a,b,c\in[n]$,
 and the permutations are listed in clockwise order around the inner hexagon
 starting and finishing there.
 The ellipsis $\cdots\cdots$ represents a sequence of $n-3$ distinct elements of
 $[n]$, the same sequence everywhere it occurs.
  A solid black curve indicates that the ratio between the
 two successive permutations is
 $\tau_n$ (so that an element jumps from position $n$ to $1$),
  while a dotted blue curve indicates $\tau_{n-2}^{-1}$ (with a jump from $1$ to $n-2$).
   } \label{qset3}
\end{figure}
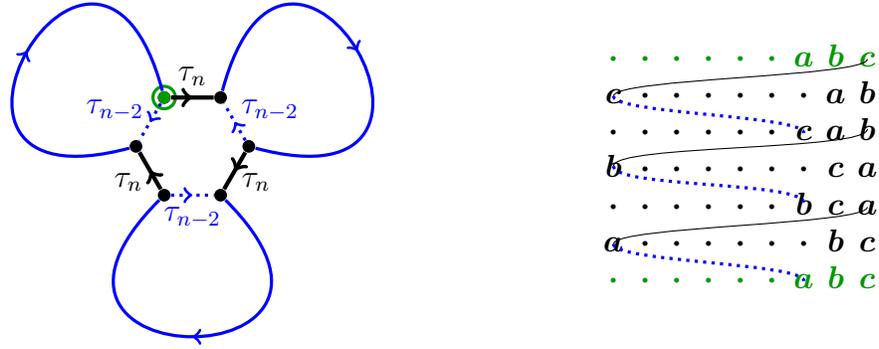

The above ideas in fact lead to the following rigorous condition for
non-existence of directed Hamiltonian cycles.  The result was proved by
Rankin \cite{rankin}, based on an 1886 proof by 
Thompson~\cite{thompson} of a special case arising in
change ringing; Swan \cite{swan} later gave a simpler
version of the proof.
\begin{thm}\label{rankin}
Consider the directed Cayley graph $\G$ of a finite group with two generators
$a,b$.  If $\order(ab^{-1})$ is odd and $|\G|/\order(a)$ is even, then $\G$ has
no directed Hamiltonian cycle.
\end{thm}
An immediate consequence is that $\A_5$ has no directed Hamiltonian cycle
(confirming the computer search result of \cite{horovitz-etzion}), and indeed
$\A_n$ has no directed Hamiltonian cycle using only two generators for odd
$n\geq 5$.

\begin{figure}
\centering
\begin{minipage}{.55\textwidth}\centering
\begin{tikzpicture}[scale=1.3]
\path[use as bounding box] (-2.3,-2.3) rectangle (2.3,2.3);
 \foreach \i in {0,...,11}
 { \draw (-15+30*\i:1) node[circle,fill,inner sep=1.5pt] (\i){};};
 \draw[very thick,darkgreen] (45:1) node[circle, draw,inner sep=3pt]{};
 \draw (45:1) node[circle,fill=darkgreen,inner sep=1.7pt]{};

 \draw[very thick,midarrow,blue,dotted] (0) to node[right] {$\tau_{n-2}$} (1);
 \draw[very thick,midarrow,red,dotted] (2) to node[above right=-1pt and -4pt] {$\tau_{n-4}$} (3);
 \draw[very thick,midarrow,red,dotted] (4) to node[above left=-1pt and -4pt] {$\tau_{n-4}$} (5);
 \draw[very thick,midarrow,blue,dotted] (6) to node[left] {$\tau_{n-2}$} (7);
 \draw[very thick,midarrow,red,dotted] (8) to node[below left=-1pt and -4pt] {$\tau_{n-4}$} (9);
 \draw[very thick,midarrow,red,dotted] (10) to node[below right=-1pt and -4pt] {$\tau_{n-4}$} (11);

 \draw[ultra thick,midarrow] (0) to node[above left=-4pt and -1pt] {$\tau_n$} (11);
 \draw[ultra thick,midarrow] (2) to node[below left=-4pt and -1pt] {$\tau_n$} (1);
 \draw[ultra thick,midarrow] (4) to node[below=1pt] {$\tau_n$} (3);
 \draw[ultra thick,midarrow] (6) to node[below right=-4pt and -1pt] {$\tau_n$} (5);
 \draw[ultra thick,midarrow] (8) to node[above right=-4pt and -1pt] {$\tau_n$} (7);
 \draw[ultra thick,midarrow] (10) to node[above=1pt] {$\tau_n$} (9);

 \draw[very thick,midarrow,blue] (1) to[bend left=135,looseness=10] (0);
 \draw[very thick,midarrow,red] (3) to[bend left=135,looseness=10] (2);
 \draw[very thick,midarrow,red] (5) to[bend left=135,looseness=10] (4);
 \draw[very thick,midarrow,blue] (7) to[bend left=135,looseness=10] (6);
 \draw[very thick,midarrow,red] (9) to[bend left=135,looseness=10] (8);
 \draw[very thick,midarrow,red] (11) to[bend left=135,looseness=10] (10);
\end{tikzpicture}
\end{minipage}%
\begin{minipage}{.4\textwidth}\centering
\begin{tikzpicture}
\boldmath
\newcommand{\bdot}{\boldsymbol\cdot}
  \matrix (T) [matrix of math nodes,nodes in empty cells,
  font=\fontsize{13}{13}\selectfont,
  row 1/.style={nodes={text=darkgreen}},
  row 13/.style={nodes={text=darkgreen}},
  column sep={13pt,between origins},row
    sep={13pt,between origins}] {
  \bdot&\bdot&\bdot&\bdot& a & b & c & d & e \\
  e & \bdot&\bdot&\bdot&\bdot & a & b & c & d \\
  \bdot&\bdot&\bdot&\bdot & a & b & e & c & d \\
  d & \bdot&\bdot&\bdot&\bdot & a & b & e & c \\
  \bdot&\bdot&\bdot&\bdot & d & a & b & e & c \\
  c & \bdot&\bdot&\bdot&\bdot & d & a & b & e \\
  \bdot&\bdot&\bdot&\bdot & c & d & a & b & e \\
  e & \bdot&\bdot&\bdot&\bdot & c & d & a & b \\
  \bdot&\bdot&\bdot&\bdot & c & d & e & a & b \\
  b & \bdot&\bdot&\bdot&\bdot & c & d & e & a \\
  \bdot&\bdot&\bdot&\bdot & b & c & d & e & a \\
  a & \bdot&\bdot&\bdot&\bdot & b & c & d & e \\
  \bdot&\bdot&\bdot&\bdot & a & b & c & d & e \\
 };
  \draw (T-1-9.center) ..controls(T-1-9.south)and(T-2-1.north).. (T-2-1.center);
  \draw[blue,very thick,dotted] (T-2-1.center)..controls(T-2-1.south)and(T-3-7.north)..(T-3-7.center);
  \draw (T-3-9.center)..controls(T-3-9.south)and(T-4-1.north)..(T-4-1.center);
  \draw[red,very thick,dotted] (T-4-1.center)..controls(T-4-1.south)and(T-5-5.north)..(T-5-5.center);
  \draw (T-5-9.center)..controls(T-5-9.south)and(T-6-1.north)..(T-6-1.center);
  \draw[red,very thick,dotted] (T-6-1.center)..controls(T-6-1.south)and(T-7-5.north)..(T-7-5.center);
  \draw (T-7-9.center)..controls(T-7-9.south)and(T-8-1.north)..(T-8-1.center);
  \draw[blue,very thick,dotted] (T-8-1.center)..controls(T-8-1.south)and(T-9-7.north)..(T-9-7.center);
  \draw (T-9-9.center)..controls(T-9-9.south)and(T-10-1.north)..(T-10-1.center);
  \draw[red,very thick,dotted] (T-10-1.center)..controls(T-10-1.south)and(T-11-5.north)..(T-11-5.center);
  \draw (T-11-9.center)..controls(T-11-9.south)and(T-12-1.north)..(T-12-1.center);
  \draw[red,very thick,dotted] (T-12-1.center)..controls(T-12-1.south)and(T-13-5.north)..(T-13-5.center);
\end{tikzpicture}
\end{minipage}
 \caption{The key construction.  \emph{Left:} replacing a suitable combination of
 generators $\tau_{n-2}$ and $\tau_{n-4}$ with $\tau_n$ links $6$ cycles into one,
 breaking the parity barrier.  We start with the $2$ blue and $4$ red thin cycles, and
 replace the dotted edges with the thick black solid edges to obtain the solid
 cycle.
 \emph{Right:} the permutations appearing
  at the vertices marked with solid discs, listed in clockwise order
  starting and ending at the
  circled vertex, which is $[{.}\,{.}\,{.}\,{.}\, ,a,b,c,d,e]$.
  The ellipsis ${\cdot}\,{\cdot}\,{\cdot}\,{\cdot}$
  represents the same sequence everywhere it occurs.} \label{qset6}
\end{figure}
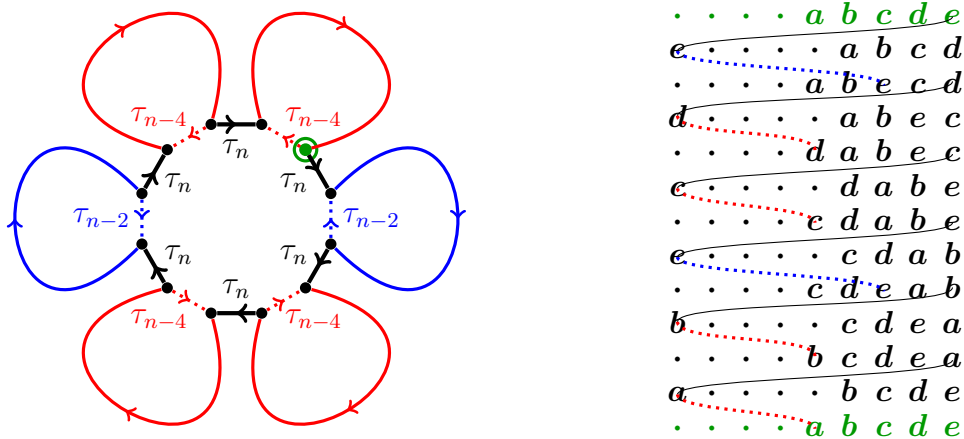
%
%
To break the parity barrier, we must use at least three generators in a
fundamental way.  The problem with the previous approach was that
$\order(\tau_\ell\tau_k^{-1})$ is odd: we need an analogous relation
involving composition of an \emph{even} number of ratios of two generators.
In terms of the graph $\A_n$, we need a cycle of length a multiple of $4$
whose edges are oriented in alternating directions.  It is clear that such a
thing must exist for all odd $n\geq 7$, because the ratios
$\tau_k\tau_\ell^{-1}$ generate the alternating group on the $n-2$ elements
$\{3,\ldots,n\}$, which contains elements of even order.  We will use the
example:
\begin{equation}
\order\bigl(\zeta\bigr)=2,\quad\text{where }
\zeta:= \tau_n\tau_{n-2}^{-1}\tau_n\tau_{n-4}^{-1}\tau_n\tau_{n-4}^{-1}.
\label{order2}
\end{equation}
It is a routine matter to check \eqref{order2}: the ratio
$\tau_n\tau_{n-s}^{-1}$ is the permutation that jumps an element from
position $n$ to $n-s$ (while fixing $1,\ldots,n-s-1$
and shifting $n-s,\ldots,n-1$ right one place), so to compute the
composition $\zeta$ of three such ratios we need only keep track of the last
$5$ elements.  \cref{qset6} (right) shows the explicit computation: starting
from an arbitrary permutation $\pi=[\ldots,a,b,c,d,e]\in A_n$, the successive
compositions
$\pi,\pi\tau_n,\pi\tau_n\tau_{n-2}^{-1},\pi\tau_n\tau_{n-2}^{-1}\tau_n,\ldots,
\pi\zeta^2=\pi$ are listed -- the ellipsis
${\cdot}\,{\cdot}\,{\cdot}\,{\cdot}$ represents the same sequence everywhere
it occurs. This explicit listing of the relevant permutations will be useful
later.

We can use the above observation to link $6$ cycles into one, as shown in
\cref{qset6} (left). Let $Q'$ be a length-$12$ cycle in $\A_n$ with edges in
alternating orientations that corresponds to the identity \eqref{order2}.
That is to say, every alternate edge in $Q'$ has label $\tau_n$, and is
oriented in the same direction around $Q'$.  The other $6$ edges are oriented
in the opposite direction, and have successive labels
$\tau_{n-2},\tau_{n-4},\tau_{n-4},\tau_{n-2},\tau_{n-4},\tau_{n-4}$.
Suppose that we start
with a cycle cover in which the two $\tau_{n-2}$-edges and the four
$\tau_{n-4}$-edges of $Q'$ all lie in distinct cycles.  Then we can delete
these $6$ edges and replace them with the six $\tau_n$-edges of $Q'$.  This
results in a new cycle cover in which these $6$ cycles have been united into
one, thus reducing the number of cycles by $5$ and changing its parity.  See \cref{qset6} (left) -- the old cycles are in thin red and blue, while the new cycle is shown by
solid lines and arcs.

%

We will prove \cref{main} by induction.  The inductive step will use one
instance of the above $6$-fold linkage to break the parity barrier, together
with many instances of the simpler $3$-fold linkage described earlier with
$(k,\ell)=(n-2,n)$. The base case $n=7$ will use the $6$-fold linkage in the
reverse direction (replacing six $\tau_n$-edges with
$\tau_{n-2},\tau_{n-4},\ldots$), together with the cases
$(k,\ell)=(7,5),(7,3)$ of the earlier linkage.

\section{Hypergraph spanning}
\label{hypergraphs}

The other main ingredient for our proof is a systematic way
of organizing the various linkages.  For this the language
of hypergraphs will be convenient.  Similar hypergraph
constructions were used in \cite{horovitz-etzion,zhang-ge}.
A \df{hypergraph} $(V,H)$ consists of a vertex set $V$ and
a set $H$ of nonempty subsets of $V$, which are called
\df{hyperedges}. A hyperedge of size $r$ is called an
$r$-hyperedge.

The \df{incidence graph} of a hypergraph $(V,H)$ is the bipartite graph with
vertex set  $V\cup H$, and with an edge between $v\in V$ and $h\in H$ if
$v\in h$. A \df{component} of a hypergraph is a component of its incidence
graph, and a hypergraph is \df{connected} if it has one component.  We say
that a hypergraph is \df{acyclic} if its incidence graph is acyclic.  Note
that this a rather strong condition: for example, if two distinct hyperedges
$h$ and $h'$ share two distinct vertices $v$ and $v'$ then the hypergraph is
not acyclic. (Several non-equivalent notions of acyclicity for hypergraphs
have been considered -- the notion we use here is sometimes called
Berge-acyclicity -- see e.g.~\cite{fagin}).

We are interested in hypergraphs of a particular kind that are related to the
linkages considered in the previous section. Let $[n]^{(k)}$ be the set of
all $n!/(n-k)!$ ordered $k$-tuples of distinct elements of $[n]$.
If $t=(a,b,c)\in [n]^{(3)}$ is a triple, define the \df{triangle}
$\T(t)=\T(a,b,c):=\{(a,b),(b,c),(c,a)\}\subset[n]^{(2)}$ of pairs that
respect the cyclic order.  (Note that $\T(a,b,c)=\T(c,a,b)\neq \T(c,b,a)$.)
In our application to Hamiltonian cycles, $\T(a,b,c)$ will encode precisely the linkage
of $3$ cycles shown in \cref{qset3}.  The following fact and its proof are
illustrated in \cref{cactus9}.

%
%
%
%
%
%
%
%
%
\begin{figure}
\centering
\includegraphics[width=\textwidth]{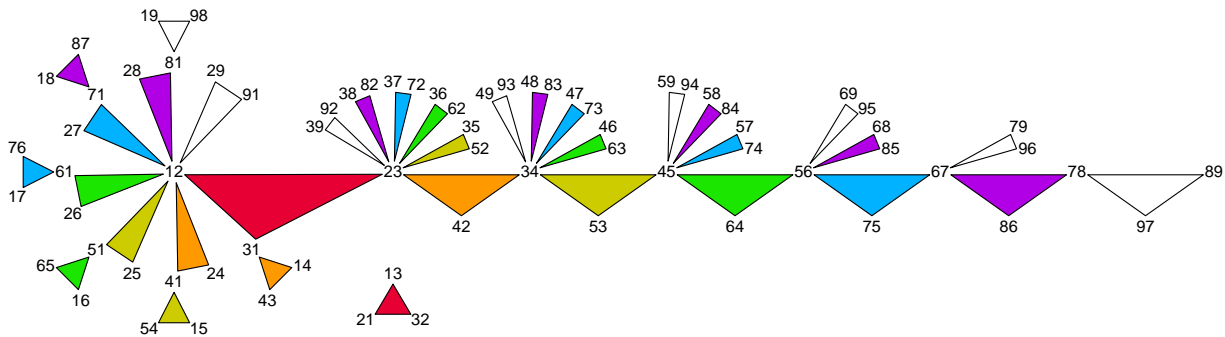}
\caption{The hypergraph of \cref{cactus}, when $n=9$.  The vertices are all the
ordered pairs $(a,b)=ab\in[n]^{(2)}$, and the hyperedges are triangles of the form $\{ab,bc,ca\}$.
Hyperedges are colored according to the step of the induction at which they are added.
In the last step from $n=8$ to $n=9$, all the white hyperedges are added,
i.e.\ those incident to vertices that contain $9$.
}
 \label{cactus9}
\end{figure}
\begin{prop}\label{cactus}
Let $n\geq 3$.  There exists an acyclic hypergraph with vertex set
$[n]^{(2)}$, with all hyperedges being triangles $\T(t)$ for $t\in
[n]^{(3)}$, and with exactly two components: one containing precisely the $3$
vertices of $\T(3,2,1)$, and the other containing all other vertices.
\end{prop}

\begin{proof}
We give an explicit inductive construction.  When $n=3$ we simply take as
hyperedges the two triangles $\T(3,2,1)$ and $\T(1,2,3)$.

Now let $n\geq 4$, and assume that $({[n-1]}^{(2)},H)$ is a hypergraph
satisfying the given conditions for $n-1$.  Consider the larger hypergraph
$([n]^{(2)},H)$ with the same set of hyperedges, and note that its components
are precisely: (i) $\T(3,2,1)$; (ii) an acyclic component which we denote $K$
that contains all vertices of $[n-1]^{(2)}\setminus\T(3,2,1)$; and (iii) the
$2n-2$ isolated vertices $\{(i,n),(n,i):i\in[n-1]\}$.

We will add some further hyperedges to $([n]^{(2)},H)$.  For $i\in[n-1]$,
write $i^+$ for the integer in $[n-1]$ that satisfies $i^+\equiv (i+1)
\bmod{(n-1)}$, and define
\begin{align*}
D:=\bigl\{&\Delta(i,i^+,n):i\in [n-1]\bigr\}\\
=\bigl\{&\T(1,2,n),\T(2,3,n),\ldots,\T(n-2,n-1,n),\;\T(n-1,1,n)\bigr\}.
\end{align*}
Any element $\Delta(i,i^+,n)$ of $D$ has $3$ vertices.  One of them, $(i,i^+)$, lies in
$K$, while the others, $(i^+,n)$ and $(n,i)$, are isolated vertices of $([n]^{(2)},H)$.
Moreover, each isolated vertex of $([n]^{(2)},H)$ appears in exactly one
hyperedge in $D$.  Therefore, $([n]^{(2)},H\cup D)$ has all the claimed
properties.
\end{proof}

We remark that the above hypergraph admits a simple (non-inductive)
description -- it consists of all $\T(a,b,c)$ such that $\max\{a,b\}<c$ and
$b\equiv (a+1) \bmod (c-1)$.

In order to link cycles into a Hamiltonian cycle we will require a \emph{connected} hypergraph. For $n\geq 3$ there is no connected acyclic hypergraph of triangles
with vertex set $[n]^{(2)}$.  (This follows from parity considerations: an
acyclic component composed of $m$ triangles has $1+2m$ vertices, but
$|[n]^{(2)}|$ is even.)  Instead, we simply introduce a larger hyperedge, as
follows.

\begin{samepage}
\begin{cor}\label{cactus2}
Let $n\geq 5$ and let $a,b,c,d,e\in[n]$ be distinct.  There exists a
connected acyclic hypergraph with vertex set $[n]^{(2)}$ such that one
hyperedge is the $6$-hyperedge $\T(a,b,e)\cup \T(c,d,e)$, and all others are
triangles $\T(t)$ for $t\in [n]^{(3)}$.
\end{cor}
\end{samepage}

\begin{proof}
By symmetry, it is enough to prove this for any one choice of $(a,b,c,d,e)$;
we choose $(2,1,4,5,3)$.  The result follows from \cref{cactus}, on noting
that $\T(3,4,5)=\T(4,5,3)$ is a hyperedge of the hypergraph constructed
there: we simply unite it with $\T(3,2,1)=\T(2,1,3)$ to form the
$6$-hyperedge.
\end{proof}

\section{The Hamiltonian cycle}
\label{cycle}

We now prove \cref{main} by induction on (odd) $n$. We give
the inductive step first, followed by the base case $n=7$.
The following simple observation will be used in the
inductive step.

\begin{lemma}\label{penultimate}
Let $n\geq 3$ be odd, and consider any Hamiltonian cycle of $\A_n$.  For
every $i\in[n]$ there exists a permutation $\pi\in A_n$ with $\pi(n)=i$ that
is immediately followed by a $\tau_n$-edge in the cycle.
\end{lemma}

\begin{proof}
Since the cycle visits all permutations of $A_n$, it must contain a directed
edge from a permutation $\pi$ satisfying $\pi(n)=i$ to a permutation $\pi'$
satisfying $\pi'(n)\neq i$.  This is a $\tau_n$-edge, since any other
generator would fix the rightmost element.
\end{proof}

\begin{proof}[Proof of \cref{main}, inductive step]
We will prove by induction on odd $n\geq 7$ the statement:
\begin{equation}
    \text{\em there exists a Hamiltonian cycle of $\A_n$
    that includes at least one
    $\tau_{n-2}$-edge.}\label{ind}
\end{equation}
 As mentioned above, we postpone the
proof of the base case $n=7$.  For distinct $a,b\in[n]$
define the set of permutations of the form $[\ldots,a,b]$:
$$A_n(a,b):=\Bigl\{\pi\in A_n: \bigl(\pi(n-1),\pi(n)\bigr)=(a,b)\Bigr\}.$$

Let $n\geq 9$, and let $L=(\tau_{s(1)},\tau_{s(2)},\ldots, \tau_{s(m)})$ be
the sequence of generators used by a Hamiltonian cycle of $\A_{n-2}$, as
guaranteed by the inductive hypothesis, in the order that they are
encountered in the cycle starting from $\id\in A_{n-2}$ (where $m=(n-2)!/2$,
and $s(i)\in\{3,5,\ldots,n-2\}$ for each $i$).   Now start from any
permutation $\pi\in A_n(a,b)$ and apply the sequence of generators $L$
(where a generator $\tau_k\in A_{n-2}$ is now interpreted as the generator $\tau_k\in A_n$ with the same name).  This gives a cycle in $\A_n$ whose vertex set is precisely $A_n(a,b)$. (The two rightmost elements $a,b$ of the permutation are undisturbed, because $L$ does not contain $\tau_n$.) Note that, for given $a,b$, different choices of the starting permutation $\pi\in A_n(a,b)$ in general result in different cycles.

We next describe the idea of the proof, before giving the details.   Consider
a cycle cover $\mathcal{C}$ comprising, for each $(a,b)\in[n]^{(2)}$, one cycle $C(a,b)$ with vertex set $A_n(a,b)$ of the form described above (so $n(n-1)$ cycles in
total).  We will link the cycles of $\mathcal{C}$ together into a single cycle by
substituting the generator $\tau_n$ at
appropriate points, in the ways discussed in \cref{parity}.  The linking
procedure will be encoded by the hypergraph of \cref{cactus2}.  The vertex
$(a,b)$ of the hypergraph will correspond to the initial cycle $C(a,b)$.
  A $3$-hyperedge $\T(a,b,c)$ will indicate a
substitution of $\tau_n$ for $\tau_{n-2}$ in $3$ of the cycles of $\mathcal{C}$,
linking them together in the manner of \cref{qset3}. The $6$-hyperedge will
correspond to the parity-breaking linkage in which $\tau_n$ is substituted
for occurrences of both $\tau_{n-2}$ and $\tau_{n-4}$, linking $6$ cycles as
in \cref{qset6}. One complication is that the starting points of the
cycles of $\mathcal{C}$ must be
chosen so that $\tau_{n-2}$- and $\tau_{n-4}$-edges occur in appropriate
places so that all these substitutions are possible. To address this, rather than choosing the cycle cover $\mathcal{C}$ at
the start, we will in fact build our final cycle sequentially, using one hyperedge at
a time, and choosing appropriate cycles $C(a,b)$ as we go.  We
will start with the $6$-hyperedge, and for each subsequent $3$-hyperedge we will
link in two new cycles.  \cref{penultimate} will ensure enough $\tau_{n-2}$-edges for subsequent steps: for any $(a,b,c)\in[n]^{(3)}$, there is a vertex of the form $[\ldots,a,b,c]$ in $C(b,c)$ followed by $\tau_{n-2}$-edge.  The inductive hypothesis \eqref{ind} will provide the $\tau_{n-4}$-edges needed for the initial $6$-fold linkage.

We now give the details. In preparation for the sequential linking procedure,
choose an acyclic connected hypergraph $([n]^{(2)},H)$ according to
\cref{cactus2}, with the $6$-hyperedge being $\T_0\cup \T_0'$, where
$\T_0:=\T(c,d,e)$ and $\T_0':=\T(a,b,e)$, and where we write
\begin{equation}
(a,b,c,d,e)=(n-4,n-3,n-2,n-1,n). \label{abcde}
\end{equation}
 Let $N=|H|-1$, and order the hyperedges
as $H=\{h_0,h_1,\ldots,h_N\}$ in such a way that $h_0=\T_0\cup \T_0'$ is the
$6$-hyperedge, and, for each $1\leq i\leq N$, the hyperedge $h_{i}$ shares
exactly one vertex with $\bigcup_{\ell=0}^{i-1} h_\ell$. (To see that this is
possible, note that for any choice of $h_0,\ldots,h_{i-1}$ satisfying this condition, connectedness of the hypergraph implies that there exists $h_i$ that shares \emph{at least} one vertex with one of its predecessors; acyclicity then implies that it
shares exactly one.)

We will construct the required Hamiltonian cycle via a sequence of steps
$j=0,\ldots, N$.  At the end of step $j$ we will have a self-avoiding
directed cycle $C_j$ in $\A_n$ with the following properties.
\begin{ilist}
  \item The vertex set of $C_j$ is the union of $A_n(x,y)$ over all
  $(x,y)\in\bigcup_{i=0}^j h_i$.
  \item \sloppy For every $(x,y,z)\in[n]^{(3)}$ such that $(y,z)\in \bigcup_{i=0}^j
  h_i$ but $\T(x,y,z)\notin \{\T_0,\T_0',h_1,h_2,\ldots,h_j\}$, there exists a permutation
  $\pi\in A_n$
  of the form $[\ldots,x,y,z]$ that is followed immediately by a
  $\tau_{n-2}$-edge in $C_j$.\fussy
\end{ilist}
We will check by induction on $j$ that the above properties hold.  The final
cycle $C_N$ will be the required Hamiltonian cycle.  The purpose of the
technical condition (ii) is to ensure that suitable edges are
available for later linkages; the idea is that the
triple $(x,y,z)$ is available for linking in two further cycles
unless it has already been used.

We will describe the cycles $C_j$ by giving their sequences of generators.
Recall that $L$ is the sequence of generators of the Hamiltonian cycle of
$\A_{n-2}$. Note that $L$ contains both $\tau_{n-2}$ and $\tau_{n-4}$, by
\cref{penultimate} and the inductive hypothesis \eqref{ind} respectively. For each of
$k=n-2,n-4$, fix some location $i$ where $\tau_k$ occurs in $L$ (so that
$s(i)=k$), and let $L[\tau_k]$ be the sequence obtained by starting at that
location and omitting this $\tau_k$ from the cycle:
$$L[\tau_k]:=\bigl(\tau_{s(j+1)},\tau_{s(j+2)}\ldots,\tau_{s(m)},\
 \tau_{s(1)},\ldots,\tau_{s(j-1)}\bigr).$$
Note that the composition in order of the elements of $L[\tau_k]$ is
$\tau^{-1}_k$.

For step $0$, let $C_0$ be the cycle that starts at $\id\in A_n$ and uses the
sequence of generators
\begin{gather*}
\tau_n, L[\tau_{n-2}],\tau_n,
L[\tau_{n-4}],\tau_n,L[\tau_{n-4}],\\
\tau_n, L[\tau_{n-2}],\tau_n, L[\tau_{n-4}],\tau_n,L[\tau_{n-4}],
\end{gather*}
(where commas denote concatenation).  This cycle is
precisely of the form illustrated in \cref{qset6} (left) by
the solid arcs and lines.  The curved arcs represent the
paths corresponding to the $L[\cdot]$ sequences.  The
vertex set of each such path is precisely $A_n(u,v)$ for
some pair $(u,v)$; we denote this path $P(u,v)$.
The solid lines represent the $\tau_n$-edges.
 Moreover, since \cref{qset6} (right)
lists the vertices (permutations) at the beginning and end
of each path $P(u,v)$, we can read off the pairs $(u,v)$. With
$a,\ldots,e$ as in \eqref{abcde}, the pairs are
$\{(d,e),(c,d),(e,c),(b,e),(a,b),(e,a)\}$. This set equals
$\T_0\cup\T_0'=h_0$, so property (i) above holds for the
cycle $C_0$.

We next check that $C_0$ satisfies (ii).  Let
$(x,y,z)\in[n]^{(3)}$ be such that $(y,z)\in h_0$.
The cycle $C_0$ includes a path $P(y,z)$ with vertex set $A_n(y,z)$
and generator sequence $L[\tau_k]$ (where $k$ is
$n-2$ or $n-4$).  Let $C(y,z)$ be the cycle that results
from closing the gap, i.e.\ appending a $\tau_k$-edge $f$ to
the end of $P(y,z)$.  Note that $P(y,z)$ and $C(y,z)$ both
have vertex set $A_n(y,z)$.  By \cref{penultimate} applied
to $\A_{n-2}$, the cycle $C(y,z)$ contains a permutation of
the form $[\ldots,x,y,z]$ immediately followed by a
$\tau_{n-2}$-edge, $g$ say.  Edge $g$ is also present
in $C_0$ unless $g=f$.  Consulting \cref{qset6}, and
again using the notation in \eqref{abcde}, we see that this happens
only in the two cases $(x,y,z)=(e,c,d),(e,a,b)$. But in
these cases we have $\T(x,y,z)=T_0,T_0'$ respectively. Thus
condition (ii) is satisfied at step $0$.

Now we inductively describe the subsequent steps.  Suppose
that step $j-1$ has been completed, giving a cycle
$C_{j-1}$ that satisfies (i) and (ii) (with parameter $j-1$
in place of $j$).   We will augment $C_{j-1}$ to obtain a
larger cycle $C_{j}$, in a manner encoded by the hyperedge
$h_j$. Let
$$h_j=\T(a,b,c)=\bigl\{(a,b),(b,c),(c,a)\bigr\}$$ (where we no longer
adopt the notation \eqref{abcde}). By our choice of the ordering of $H$,
exactly one of these pairs belongs to $\bigcup_{i=0}^{j-1} h_i$; without loss
of generality, let it be $(b,c)$. By property (ii) of the cycle $C_{j-1}$, it
contains a vertex of the form $[\ldots,a,b,c]$ immediately followed by a
$\tau_{n-2}$-edge, $f$ say.  Delete edge $f$ from $C_{j-1}$ to obtain a
directed path $P_{j-1}$ with the same vertex set. Append to $P_{j-1}$ the
directed path that starts at the endvertex of $P_{j-1}$ and then uses the
sequence of generators
$$\tau_n,L[\tau_{n-2}],\tau_n,L[\tau_{n-2}],\tau_n.$$
Since $\order(\tau_n\tau_{n-2}^{-1})=3$, this gives a
cycle, which we denote $C_j$.

The new cycle $C_j$ has precisely the form shown in
\cref{qset3} (left) by the solid arcs and lines, where
$C_{j-1}$ is the thin blue cycle
in the upper left, containing the circled vertex, which is
the permutation $[\ldots,a,b,c]$.  The arc is $P_{j-1}$, and the dotted edge
is $f$.  As before, the
permutations at the filled discs may be read from
\cref{qset3} (right).  Thus, $C_j$ consists of the path
$P_{j-1}$, together with two paths $P(a,b),P(c,a)$ with
respective vertex sets $A_n(a,b),A_n(c,a)$ (the other two
thin blue arcs in the figure), and three $\tau_n$-edges (thick black lines)
connecting these three paths. Hence $C_j$ satisfies
property (i).

We now check that $C_j$ satisfies (ii).  The argument is similar
to that used in step $0$.  Let $(x,y,z)$ satisfy the
assumptions in (ii).  We consider two cases.  First suppose
$(y,z)\notin h_j$.  Then $(y,z)\in \bigcup_{i=0}^{j-1} h_i$, and
so property (ii) of $C_{j-1}$ implies that $C_{j-1}$ has a vertex of the form
$[\ldots,x,y,z]$ followed by a $\tau_{n-2}$-edge $g$, say.
Then $g$ is also present in $C_j$ unless $g=f$.  But in that case we have $(x,y,z)=(a,b,c)$, and so $\T(x,y,z)=h_j$,
contradicting the assumption on $(x,y,z)$.
On the other hand, suppose $(y,z)\in h_j$. Then $(y,z)$ equals
$(a,b)$ or $(c,a)$.  Suppose the former; the argument in
the latter case is similar. Let $C(a,b)$ be the cycle
obtained by appending a $\tau_{n-2}$-edge to $P(a,b)$.
Applying \cref{penultimate} shows that $C(a,b)$ contains a vertex of
the form $[\ldots,x,a,b]$ followed by a $\tau_{n-2}$-edge
$g$, say. Then $g$ is also present in $P(a,b)$ unless $x=c$,
but then $\T(x,y,z)=h_j$, contradicting the assumption in
(ii). Thus, property (ii) is established.

To conclude the proof, note that the final cycle $C_N$ is
Hamiltonian, by property (i) and the fact that the hypergraph of \cref{cactus2} has vertex set $[n]^{(2)}$.  To check that it includes some
$\tau_{n-2}$-edge as required for \eqref{ind},
recall that $h_N$ has only one vertex in
common with $h_0,\ldots,h_{N-1}$, so there exist $x,y,z$
with $(y,z)\in h_N$ but $\T(x,y,z)\notin H$.  Hence property (ii)
implies that $C_N$ contains a $\tau_{n-2}$-edge.
\end{proof}

\begin{proof}[Proof of \cref{main}, base case]
For the base case of the induction, we give an explicit
directed Hamiltonian cycle of $\A_7$ that includes $\tau_5$
at least once.  (In fact the latter condition must
necessarily be satisfied, since, as remarked earlier,
\cref{rankin} implies that there is no Hamiltonian cycle
using only $\tau_3$ and $\tau_7$.)

\begin{table}
{\renewcommand{\arraystretch}{1.25}
\begin{tabular}{|c|c|c|}
  \hline
row & permutations & generator \\
  \hline
1 & $6\overline{7} \overline{7}\overline{7}\s\s\s,
\overline{7}\overline{7}\overline{7} 6 \s\s\s$ & $\tau_5$ \\
2 & $67\s\s\s\s\s,76\s\s\s\s\s$ & $\tau_3$ \\
3 & $567\overline{1}\s\s\s, 576\s\s\s\s$ & $\tau_5$ \\
4 & $2567\s\s\s,4576\s\s\s$ & $\tau_5$ \\
5 & $5671234,5612347,5623714,5637142$ & $\tau_3$ \\
6 & $5623471,5671423$ & $\tau_5$  \\
7 & otherwise & $\tau_7$\\
  \hline
\end{tabular}
\vspace{3pt}
} \caption{Rules for generating a directed Hamiltonian
cycle of $\A_7$. Permutations of the given forms should be
followed by the generator in the same row of the table. The
symbol $\s$ denotes an arbitrary element of $[7]$, and
$\overline{a}$ denotes any element other than
$a$.}\label{a7}
\end{table}
\cref{a7} specifies which generator the cycle uses
immediately after each permutation of $A_7$, as a function
of the permutation itself. The skeptical reader may simply
check by computer that these rules generate the required
cycle. But the rules were constructed by hand; below we
briefly explain how.

First suppose that from every permutation of $A_7$ we apply
the $\tau_7$ generator, as specified in row 7 of the table.
This gives a cycle cover comprising $|A_7|/7=360$ cycles of
size $7$. Now consider the effect of replacing some of
these $\tau_7$'s according to rows 1--6 in succession. Each
such replacement performs a linkage, as in
\cref{qset3,qset6}. Row 1 links the cycles in sets of $3$
to produce $120$ cycles of length $21$, each containing
exactly one permutation of the form $67\s\s\s\s\s$ or
$76\s\s\s\s\s$. Row 2 then links these cycles in sets of
$5$ into $24$ cycles of length $105$, each containing
exactly one permutation of the form $675\s\s\s\s$ or
$765\s\s\s\s$. Rows $3$ and $4$ link various sets of three
cycles, permuting elements $1234$, to produce $6$ cycles.
Finally, rows $5$ and $6$ break the parity barrier as
discussed earlier, uniting these $6$ cycles into one.
\end{proof}

\section{Even size}

We briefly discuss a possible approach for even $n$. Recall
that $M_n$ is the maximum length of a cycle $\S_n$ in which
no two permutations are related by an adjacent
transposition.

To get a cycle longer than $M_{n-1}$ we must use $\tau_n$. But this is an odd
permutation, so we cannot remain in the alternating group $A_n$.  We suggest
following $\tau_n$ immediately by another odd generator, say $\tau_{n-2}$, in
order to return to $A_n$ (note that $\tau_2$ is forbidden).  In order to
include permutations of the form $[\ldots,j]$ for every $j\in[n]$, we need to
perform such a transition (at least) $n$ times in total in our cycle.  In the
$i$th transition we visit one odd permutation, $\alpha_i$ say, between the
generators $\tau_n$ and $\tau_{n-2}$.  For the remainder of the cycle we
propose using only generators $\tau_k$ for odd $k$, so that we remain in
$A_n$.

In fact, one may even try to fix the permutations $\alpha_1,\ldots,\alpha_n$ in
advance.  The problem then reduces to that of finding long self-avoiding
directed paths in $\A_{n-1}$, with specified start and end vertices, and
avoiding certain vertices -- those that would result in a permutation that is
related to some $\alpha_i$ by an elementary transposition.  Since there are $n$
$\alpha_i$'s and $n-1$ elementary transpositions, there are $O(n^2)$ vertices
to be avoided in total.

Since, for large $n$, the number of vertices to be avoided is much smaller than $|A_{n-1}|$, we think it very likely that paths of length
$(1-o(1))|A_{n-1}|$ exist, which would give $M_n\geq (1-o(1))n!/2$ as
$n\to\infty$.  It is even plausible that $M_n\geq n!/2-O(n^2)$ might be
achievable.  The graph $\A_{n-1}$ seems to have a high degree of global connectivity,
as evidenced by the diverse constructions of cycles of close to optimal
length in \cite{horovitz-etzion,yehezkeally-schwartz,zhang-ge}.  For a
specific approach (perhaps among others),
one might start with a short path linking the required
start and end vertices, and then try to successively link in short cycles
(say those that use a single generator such as $\tau_{n-1}$) in the manner
of \cref{qset3}, leaving out the relatively few short cycles that contain
forbidden vertices.  It is conceivable that the forbidden permutations might
conspire to prevent such an approach, for example by blocking even short paths between the start and end vertices.  However, this appears unlikely, especially given the additional flexibility in the choice of $\alpha_1,\ldots,\alpha_n$.

While there appear to be no fundamental obstacles, a proof for general even
$n$ along the above lines might be rather messy.  (Of course, this does not preclude some other more elegant approach).  Instead, the approach was combined with a
computer search to obtain a cycle of length $315(=6!/2-45)$ for $n=6$, which
is presented below, answering a question of \cite{horovitz-etzion}, and
improving the previous record $M_6\geq 57$ \cite{horovitz-etzion} by more
than a factor of $5$.  The case $n=6$ is in some respects harder than larger
$n$: the forbidden vertices form a larger fraction of the total, and $\A_5$
has only two generators, reducing available choices.  (On the other hand, the
search space is of course relatively small).  Thus, this result also lends support
to the belief that $M_n\geq (1-o(1))n!/2$ as $n\to\infty$.

The search space was reduced by quotienting the graph $\S_6$ by a group of
order $3$ to obtain a Schreier graph, giving a cycle in which the sequence of
generators is repeated $3$ times. The cycle uses the sequence of generators
$(\tau_{k(i)})$ where $(k(i))_{i=1}^{315}$ is the sequence
\newcommand{\z}{\hat{3}}
\begin{align*}
\bigl(
&\mathtt{64\ 55\z5\z\z5555\z555\z555\z\z5555\z5555\z555\z5555\z\z555\z5555\z\z} \\
&\mathtt{64\ 555\z\z5\z55\z\z55\z\z5\z555\z5555\z555\z555\z\z5555\z555\z5555\z\z5}
\bigr)^3.
\end{align*}
(Here, commas are omitted, the superscript indicates that the sequence is repeated three times, and $3$'s are marked as an aid to visual clarity).

\bibliographystyle{abbrv}
\bibliography{bib}

\end{document}